\documentclass[oneside]{amsart}
\usepackage{amsmath,amssymb,amsfonts,amsthm}
\usepackage{color}
\usepackage{graphicx}
\usepackage{geometry} 
\usepackage{amscd}
\usepackage{enumitem}

\title{One form deformation of sprays}

\author[Elgendi]{S. G. Elgendi}

\address{S.G.~Elgendi, Department of Mathematics, Faculty of Science, Benha
  University, Egypt} \email{salah.ali@fsci.bu.edu.eg, \,\,\, salahelgendi@yahoo.com}



\thanks{}

\def\blue#1{\textcolor[rgb]{0.0,0.0,1.0}{#1}}

\newcommand{\T}{{\mathcal T}}
\newcommand{\C}{{\mathcal C}}

\newcommand{\Real}{\mathbb R}

\newcommand{\To}{\longrightarrow}

\newcommand{\tm}{\T M}

\newcommand{\TM}{\mathcal T\hspace{-1pt}M}

\def\+{\!+\!}

\def\={\!=\!}
\def\<{\!<\!}
\def\>{\!>\!}

 \let\oldmarginpar\marginpar
\renewcommand\marginpar[1]{\oldmarginpar[\raggedleft\footnotesize #1]%
  {\blue{\raggedright \footnotesize \fbox{
      \begin{minipage}{1.0\linewidth}
        #1
      \end{minipage}
}}}}

\numberwithin{equation}{section} 
\numberwithin{figure}{section} 

\theoremstyle{plain}
\newtheorem*{theorem*}{Theorem}
\newtheorem{theorem}{Theorem}[section]
\newtheorem{lemma}[theorem]{Lemma}
\newtheorem{proposition}[theorem]{Proposition}
\newtheorem{corollary}[theorem]{Corollary}
\theoremstyle{definition}
\newtheorem{definition}[theorem]{Definition}

\theoremstyle{remark}
\newtheorem{example}{Example}

\newtheorem*{acknowledgement*}{Acknowledgement}

\begin{document}

\maketitle
\begin{abstract}
  In this paper, we introduce the notion of  one form deformation of  sprays.  The metrizability of the new spray, when the background spray is flat,   is characterized. Therefore, we obtain new projectively flat metrics of constant flag curvature $1$. Moreover, these new metrics are not,    generally, isometric to the Klein metric via affine  transformations. New  solutions for Hilbert's   fourth problem are obtained and constructed. Various examples are discussed and studied.
 \end{abstract}

\section{Introduction}

 The notion of sprays was introduced by  W. Ambrose et al. \cite{sprays} in  1960.  A system of second order ordinary differential equations (SODE) with positively $2$-homogeneous coefficients functions  can be shown as a second order vector field, which is called a spray.  All sprays are associated with a SODE and conversely, a spray can be associated with a SODE.  If such a system introduces  the variational (Euler-Lagrange) equations of the energy of a Finsler metric, then it is said to be Finsler metrizable and in this case the  spray is the geodesic spray of the Finsler metric. The Finsler metrizability problem for a spray $S$ looking for a Finsler structure whose geodesics coincide with the geodesics of $S$. The metrizability problem can be considered as a special case of the inverse problem of the calculus of variation. Several interesting results on the metrizability problem can be found in the literature,  we refer, for example,  to \cite{Bucataru2, crampin, Krupka, Krupkova,  Szilasi} and the references therein.  
 
The geodesics of a Finsler structure $F$ on an open subset $U \subset \mathbb{R}^n$ are straight lines if and only if the spray coefficients of $F$ are  given in the form $G^i=P(x,y)y^i$. Straight lines in $U$ are parametrized by $\sigma(t) = f(t)a + b$, where $a, b \in \Real^n$ are constant vectors and $f(t) > 0$ is a positive function. The regular case of Hilbert's Fourth Problem is to  characterize all locally  projectively flat Finsler metrics; that is, the metrics whose  geodesics are straight lines  on an open subset of $\Real^n$.  Beltrami's theorem states that a Riemannian metric is locally projectively flat if and only if it has constant sectional curvature. In  Finslerian case, this is not true. There are non projectively flat Finsler metrics of constant flag curvature. Flag curvature is an analogue of sectional curvature in Finsler geometry.

 Bucataru and Muzany (\cite{Bucataru3,Bucataru4})  characterized   the sprays which are  metrizable by Finsler metrics of  constant    flag curvature $\kappa$.

 In this paper, we introduce the   one form deformation of sprays. For a given spray $S$ on a manifold $M$, we define the one form deformation as the projective deformation of $S$ by a one form $\beta$ on $M$. In other words we get a new spray $\widetilde{S}=S-2 \beta \mathcal{C}$ which has the same geodesics of $S$.  We focus our attention to the  one form deformation of a flat spray; namely, ${S}=S_0-2\beta \mathcal{C}$ and $\beta(x,y)=y^kb_k(x)$ is a  one form on the manifold $M$.  We study the Finsler metrizability  of the  deformation  spray ${S}$. We characterize the metrizability of $S$ by a Finsler metric of constant flag curvature.    We obtain a new family of projectively flat   metrics of constant flag curvature and hence  new solutions for Hilbert's fourth problem.
 
 The  metric on $\mathbb{B}^n\subset\mathbb{R}^n$  given by

$$ F_{\mu}=\sqrt{\frac{(1+\mu|x|^2)|y|^2-\mu\langle x,y\rangle^2}{(1+\mu|x|^2)^2}}, \qquad y\in T_x\mathbb{B}^n\simeq \mathbb{R}^n $$
 is projectively flat Riemannian metrics of constant (flag) curvature $\mu$  with the projective factor $P=\frac{-\mu\langle x,y \rangle}{1+\mu|x|^2}$. It is known that  every locally projectively flat Riemannian metric is locally isometric to $F_{\mu}$ for some constant $\mu$.  In this paper, we obtained new family of projectivley flat metrics. By calculating the Jacobi endomorphism, we conclude that this it has constant flag curvature $1$. This  family is given by
$$ 
F=\sqrt{\frac{4h(x)c_{ij}y^iy^j-4(c_{ij}x^iy^j)^2-4\langle c',y\rangle c_{ij}x^iy^j-\langle c',y\rangle^2} {(2(h(x)))^2}}
$$
and its projective factor is
$$P(x,y)=-\frac{2c_{ij}x^iy^j+\langle c',y\rangle} {2(h(x))},$$
where $h(x):=c_{ij}x^ix^j+\langle c',x\rangle+c$, $c_{ij}=c_{ji}$,  $c$, $c'=(c_1, c_2, ..., c_n)$ are  constants.
 Starting by $F_\mu$ (when $\mu=1$), the general  transformation, $x$ to $Ax+B$ and $y$ to $Ay$ where $A$ is an $n\times n$ invertible matrix and $B$ is an arbitrary $n\times1$ matrix, generates projectively flat Riemannian metrics. The obtained class of projectively flat metrics is not, generally,  isometric to $F_\mu$ metric via affine transformations.

Since the deformation spray $S$ of a flat spray $S_0$ is always isotropic and in the case that the curvature of $S$ is non zero, then the metric freedom \cite{Mu-Elgendi} of $S$ is unique up to some constants. Hence, in our case the deformation of a flat spray  by the specific  one form  is metrizable by unique metric. However, we construct   new projectively flat Finsler metrics and hence Finsler solutions for Hilbert's fourth problem.  
 
 It is known that, \cite{Bucataru3},  one of the conditions for   a spray $S$  with non-vanishing Ricci curvature to be  metrizable by a Finsler function of non-zero constant flag curvature is  $\text{rank}\ dd_J(\text{Tr} \ \Phi)=2n$.
As an application of the deformation of a flat spray by a  one form, we answer the following question: \\

\textit{ Does  any spray of non-vanishing Ricci curvature  satisfy the condition $\text{rank}\ dd_J(\text{Tr} \ \Phi)=2n$?}\\
 
By an example, we show that for a spray $S$,  if  $S$ has non vanishing Ricci curvature, then the  rank of the form $ dd_J(\text{Tr} \ \Phi)$ is not necessarily maximal; that is, the condition $\text{rank}\ dd_J(\text{Tr} \ \Phi)=2n$ is sharp for the metrizability of $S$.

\section{Preliminaries}

Let $M$ be an $n$-dimensional manifold and $(TM,\pi_M,M)$ be its tangent bundle
and $(\T M,\pi,M)$ the subbundle of nonzero tangent vectors.  We denote by
$(x^i) $ local coordinates on the base manifold $M$ and by $(x^i, y^i)$ the
induced coordinates on $TM$.  The vector $1$-form $J$ on $TM$ defined,
locally, by $J = \frac{\partial}{\partial y^i} \otimes dx^i$ is called the
natural almost-tangent structure of $T M$. The vertical vector field
$\C=y^i\frac{\partial}{\partial y^i}$ on $TM$ is called the canonical or the
Liouville vector field.

A vector field $S\in \mathfrak{X}(\T M)$ is called a spray if $JS = \C$ and
$[\C, S] = S$. Locally, a spray can be expressed as follows
\begin{equation}
  \label{eq:spray}
  S = y^i \frac{\partial}{\partial x^i} - 2G^i\frac{\partial}{\partial y^i},
\end{equation}
where the \emph{spray coefficients} $G^i=G^i(x,y)$ are $2$-homogeneous
functions in the $y=(y^1, \dots , y^n)$ variable. A curve $\sigma : I \rightarrow M$ is called regular if $\sigma' : I \rightarrow \tm$, where $\sigma'$ is the tangent lift of $\sigma$. A regular curve
 $\sigma $ on $M$ is called  \emph{geodesic} of a spray $S$ if
$S \circ \sigma' = \sigma''$. Locally, $\sigma(t) = (x^i(t))$ is a geodesic of $S$ if and
only if it satisfies the equation
\begin{equation}
  \label{eq:sode}
  \frac{d^2 x^i}{dt^2} +2G^i\Big(x ,\frac{dx}{dt}\Big)=0.
\end{equation}

An orientation preserving reparameterization $t \rightarrow \tilde{t}(t)$ of the system \eqref{eq:sode} leads to a new spray $\widetilde{S} = S - 2PC$. The scalar function $P \in C^\infty(\tm)$ is $1$-homogeneous and it is related to the new parameter by
\begin{equation}
  \label{P-factor}
  \frac{d^2 \tilde{t}}{dt^2} =P\Big(x^i(t) ,\frac{dx^i}{dt}\Big) \frac{d \tilde{t}}{dt}, \,\, \frac{d \tilde{t}}{dt}>0.
\end{equation}
\begin{definition}
Two sprays $S$ and $\widetilde{S}$ are projectively related if their geodesics coincide up to an orientation preserving reparameterization. $\widetilde{S}$ is called the projective deformation of spray $S$.
\end{definition}

A nonlinear connection is defined by an $n$-dimensional distribution $H : u \in \tm \rightarrow H_u\in T_u(\tm)$ that is supplementary to the vertical distribution, which means that for all $u \in \tm$, we have 
$T_u(\tm) = H_u(\tm) \oplus V_u(\tm).$

Every spray S induces a canonical nonlinear connection through the corresponding horizontal and vertical projectors,
\begin{equation}
  \label{projectors}
    h=\frac{1}{2}  (Id + [J,S]), \,\,\,\,            v=\frac{1}{2}(Id - [J,S])
\end{equation}
Equivalently, the canonical nonlinear connection induced by a spray can be expressed in terms of an almost product structure  $\Gamma = [J,S] = h - v$. With respect to the induced nonlinear connection, a spray $S$ is horizontal, which means that $S = hS$. Locally, the two projectors $h$ and $v$ can be expressed as follows
$$h=\frac{\delta}{\delta x^i}\otimes dx^i, \quad\quad v=\frac{\partial}{\partial y^i}\otimes \delta y^i,$$
$$\frac{\delta}{\delta x^i}=\frac{\partial}{\partial x^i}-N^j_i(x,y)\frac{\partial}{\partial y^j},\quad \delta y^i=dy^i+N^j_i(x,y)dx^i, \quad N^j_i(x,y)=\frac{\partial G^j}{\partial y^i}.$$
The Jacobi endomorphism is defined by
$$\Phi=v\circ [S,h]=R^i_{\,\,j}\frac{\partial}{\partial y^i}\otimes dx^j=\left(2\frac{\partial G^i}{\partial x^j}-S(N^i_j)-N^i_kN^k_j \right)\frac{\partial}{\partial y^i}\otimes dx^j.$$
The two curvature tensors are related by
$$3R=[J,\Phi], \quad \Phi=i_SR. $$
The Ricci curvature, $\text{Ric}$, and the Ricci scalar, $\rho\in C^\infty(\TM)$  \cite{Bao-Robles} and \cite{shen-book1},  are given by
$$\text{Ric}=(n-1)\rho=R^i_{\,\,i}=\text{Tr}(\Phi).$$

\begin{definition} 
  A spray $S$ is called \emph{isotropic} if  the Jacobi endomorphism has the form
  \begin{displaymath}
    \Phi=\rho J-\alpha\otimes \C,
  \end{displaymath}
  where $\alpha$  is a semi-basic $1$-form $\alpha\in\Lambda^1(\T M)$.
\end{definition}
Due to the homogeneity condition, for isotropic sprays, the Ricci scalar is given by $\rho=i_S\alpha$.

\begin{definition}
A Finsler manifold  of dimension $n$ is a pair $(M,F)$, where $M$ is a  differentiable manifold of dimension $n$ and $F$ is a map  $$F: TM \To \Real ,\vspace{-0.1cm}$$  such that{\em:}
 \begin{description}
    \item[(a)] $F$ is smooth and strictly positive on $\T M$ and $F(x,y)=0$ if and only if $y=0$,
    \item[(b)]$F$ is positively homogenous of degree $1$ in the directional argument $y${\em:}
    $\mathcal{L}_{\mathcal{C}} F=F$,
    \item[(c)] The metric tensor $g_{ij}=\frac{\partial^2E}{\partial y^i\partial y^j}$ has rank $n$ on $\T M$, where $E:=\frac{1}{2}F^2$ is the energy function.
 \end{description}
 \end{definition}
 
 Since the $2$-form $dd_JE$ is
non-degenerate,  the Euler-Lagrange equation
\begin{equation}
  \label{eq:EL}
  \omega_E:=i_Sdd_JE-d(E-\mathcal L_{{C}}E)=0
\end{equation}
uniquely determines a spray $S$ on $TM$.  This spray is called the
\emph{geodesic spray} of the Finsler function.  The $\omega_E$ is called the
Euler-Lagrange form associated to $S$ and $E$.

\begin{definition}
  A spray $S$ on a manifold $M$ is called \emph{Finsler metrizable} if there
  exists a Finsler function $F$ such that the geodesic spray of the Finsler
  manifold $(M,F)$ is $S$.
\end{definition}
 
 \begin{definition} 
 The function $F$ is said to be of scalar flag curvature if there exists a function $\kappa\in C^\infty(\TM)$ such that
  \begin{displaymath}
    \Phi=\kappa(F^2 J-Fd_JF\otimes \C).
  \end{displaymath}
\end{definition}
It follows that for a Finsler function F, of scalar flag curvature $\kappa$, its geodesic spray S is isotropic, with Ricci scalar $\rho=\kappa F^2$ and the semi-basic 1-form $\alpha = \kappa Fd_JF$.

 \begin{definition} 
 A Finsler metric $F = F(x,y)$ on an open subset $U \subset
\Real^n$ is said to be projectively flat if all geodesies are straight lines in $U$. A Finsler metric $F$
on a manifold $M$ is said to be locally projectively flat if at any point, there
is a local coordinate system $(x^i)$ in which $F$ is projectively flat.
\end{definition}

From now on, we use the notations $\partial_i$ for the  partial differentiation with respect to $x^i$ and 
   $\dot{\partial}_i$ for the  partial differentiation    with respect to  $y^i$.
   
 By  \cite{Hamel}, a Finsler metric $F $ on an open subset
$U \subset \Real^n $ is projectively flat if and only if it satisfies the following system of
equations,
\begin{equation}
y^j\dot{\partial}_i\partial_jF-\partial_iF=0,
\end{equation}
 In this case,  $G^i=Py^i$ where $P=P(x,y)$, the projective factor of $F$, given by $P=\frac{\partial_kFy^k}{2F}$.

\section{One form deformation}

For a given  spray $S$ on a manifold $M$,  we mean by the one form deformation of $S$ the  projective deformation of  $S$  defined by
$$\widetilde{S}=S-2\beta \C,$$
where $\beta$ is a one form on the manifold $M$. 
In this paper, we will focus our attention to the one form deformation of flat sprays. Let $S=S_0-2\beta \C$, where $S_0$ is a flat spray; that is, $S_0=y^i\partial_i$ and  $\beta=b_i(x)y^i$ is one form on $M$, then by \cite{Bucataru1}, one has the following
\begin{lemma}\label{lemma1}
For the deformation spray $S=S_0-2\beta\C$ of a flat spray $S_0$.
 The corresponding horizontal projectors and  Jacobi endomorphisms of the two sprays are
related as follows:
\begin{description}

\item[(a)] $h = h_0-\beta J-d_J\beta\otimes \C$,


\item[(b)] $\Phi = (\beta^2 - S_0\beta)J- (\beta d_J\beta + d_J (S_0\beta) - 3d_{h_0} \beta) \otimes \C$,

\end{description}
\end{lemma}

In \cite{Bucataru4}, Bucataru and Muzsnay characterized the metrizability of projective deformation of flat sprays as follows.

\begin{theorem}\label{MuBu}
 \cite{Bucataru4}   The spray $S=S_0-2P \C$ is   metrizable by a Finsler function of non zero constant flag curvature if and only if
\begin{description}
  \item[(i)] $d_J\alpha=0$,  
  \item[(ii)] $d_h\rho=0$, 
  \item[(iii)] $\rm{rank}(dd_J\rho)=2n$,
\end{description}
where $\alpha$  is a semi-basic 1-form given by $\alpha= P d_JP +d_J(S_0P)-3d_{h_{0}}P$ and $\rho$ is the Ricci scalar given by $\rho=P^2-S_0P$.
\end{theorem}

\begin{proposition}
Let $S=S_0-2\beta \C$ be a  one form deformation of a flat spray $S_0$ with non-vanishing Ricci curvature.
Then, necessary conditions for the  properties $d_J\alpha=0$ and $ \text{rank}\ dd_J(\text{Tr} \ \Phi)=2n$
 to be hold are
\begin{description}
  \item[(a)] $\partial_ib_j-\partial_jb_i=0$, i.e $b_i$ is gradient ($\beta$ is closed on $M$).
  \item[(b)] $\det(\partial_ib_j+b_ib_j)\neq 0$.
\end{description}
\end{proposition}
\begin{proof}
Let $d_J\alpha=0$,  since $d_J\alpha=-3d_{h_0}d_J\beta$, then  we have
$$d_{h_0}d_J\beta(\partial_i,\partial_j)=\partial_i\dot{\partial}_j\beta-\partial_j\dot{\partial}_i\beta=0.$$
Using the fact $\dot{\partial}_j\beta=b_j$, we get $\partial_ib_j-\partial_jb_i=0$, i.e $b_i$ is gradient.

Since $\text{Tr} \ \Phi=(n-1)(\beta^2-S_0\beta)$, then by a direct calculations and using that $b_i$ is gradient, we obtain
$$dd_J(\text{Tr} \ \Phi)=2(n-1)((\partial_ib_j+b_ib_j)dx^i\wedge dy^j+(b_i\partial_j\beta-b_j\partial_i\beta)dx^i\wedge dx^j).$$
Consequently, $\text{rank}\ dd_J(\text{Tr} \ \Phi)=2n$ if $\det{(\partial_ib_j+b_ib_j)}\neq 0$.
\end{proof}

The above proposition together with  \cite{Bucataru3} show the following
\begin{corollary}
 Let$S=S_0-2\beta \C$ be a  one form deformation of a flat spray $S_0$ with non-vanishing Ricci curvature.
Then, necessary  conditions for $S$ to be metrizable are 
\begin{description}
  \item[(a)] $\partial_ib_j-\partial_jb_i=0$, i.e $b_i$ is gradient,
  \item[(b)] $\det(\partial_ib_j+b_ib_j)\neq 0$.
\end{description}
\end{corollary}

Now, we are in a position to announce and prove the first result of this work.
\begin{theorem}
The  one form deformation  $S=S_0-2\beta \C$, $\beta(x,y)=y^kb_k(x)$, of a flat spray $S_0$, is  Finsler metrizable   if and only if
\begin{equation}
   \label{b_i(x)}
   b_k(x)=-\frac{2c_{ik}x^i+c_k} {2(c_{ij}x^ix^j+\langle c',x\rangle+c)},
\end{equation}
where
$c_{ij}=c_{ji}$,  $c$, $c'=(c_1, c_2, ..., c_n)$ are  constants.
\end{theorem}
\begin{proof}  
We are going to prove  that the conditions \textbf{(i)-(iii)} in Theorem \ref{MuBu} are satisfied if and only if $b_k(x)$ given by the form \eqref{b_i(x)}. 

Assume that    $S=S_0-2\beta \C$, where $\beta=y^kb_k$, $b_k(x)$ is given by \eqref{b_i(x)}. 
Since $d_J\alpha=-3d_{h_0}d_J\beta$, then  we have
$$d_{h_0}d_J\beta(\partial_i,\partial_j)=\partial_i\dot{\partial}_j\beta-\partial_j\dot{\partial}_i\beta=\partial_ib_j-\partial_jb_i.$$
Using the property  that $c_{ij}$ is symmetric,  $b_i$ is gradient and therefore $d_J\alpha=0$.

To calculate $\rho$, let's compute  $S_0(\beta)=y^k\partial_k\beta$,
\begin{equation} 
  \label{s0.beta}
   S_0(\beta)=-\frac{2h(x)c_{ij}y^iy^j-\left(2c_{ij}x^iy^j+\langle c',y\rangle\right)^2} {2(h(x))^2},
\end{equation}
for simplicity, we use $h(x):=c_{ij}x^ix^j+\langle c',x\rangle+c$.
Using the formula of $\beta$  together with (\ref{s0.beta}), we have
\begin{equation}
  \label{rho}
  \rho=\frac{4h(x)c_{ij}y^iy^j-4(c_{ij}x^iy^j)^2-4\langle c',y\rangle c_{ij}x^iy^j-\langle c',y\rangle^2} {(2h(x))^2}.
\end{equation}
Differentiating (\ref{rho}) with respect to $\partial_k$ and  $\dot{\partial}_k$, we obtain:
 \begin{equation}
   \label{pa-k-rho}
   \begin{array}{rcl}
   \partial_k\rho&=&\frac{4c_{ij}y^iy^j\left(2c_{rk}x^r+c_k\right)-8c_{ij}x^iy^jc_{rk}y^r-4\langle c',y\rangle c_{ik}y^i}{(2h(x))^2}   \\
   && - \frac{4\left(2c_{rk}x^r+c_k\right)\left(4h(x)c_{ij}y^iy^j-4(c_{ij}x^iy^j)^2-4\langle c',y\rangle c_{ij}x^iy^j-\langle c',y\rangle^2\right)}{(2h(x))^3},
   \end{array}
 \end{equation}
 \begin{equation}
    \label{pa-dot-k-rho}
    \begin{array}{rcl}
    \dot{\partial}_k\rho&=&\frac{8h(x)c_{ik}y^i-8c_{ij}x^iy^jc_{kr}x^r-4c_kc_{ij}x^iy^j-4\langle c',y\rangle c_{kj}x^j
    -2\langle c',y\rangle c_k}{(2h(x))^2}.
   \end{array}
 \end{equation}
 By using Lemma \ref{lemma1} (a), we have $d_h\rho=d_{h_0}\rho - \beta d_J\rho-2\rho d_{J}\beta$
 which given, locally,  by
 $$d_h\rho(\partial_i)=\partial_i\rho - \beta \dot{\partial}_i\rho-2\rho b_i.$$ 
  Now, substituting from  (\ref{rho}), (\ref{pa-k-rho}) and (\ref{pa-dot-k-rho}) into  the above equation, we get $d_h\rho(\partial_i)=0$.

Putting $\rho_{ij}:=\dot{\partial}_i\dot{\partial}_j\rho$, then we have
$$   \rho_{ij}=\frac{4c_{ij}h(x)-4(c_{ir}x^k)(c_{jk}x^k)-2c_ic_{jk}x^k
  -2c_jc_{ik}x^k-c_ic_j}{(2h(x))^2}.$$
The condition \textbf{(iii)} (regularity condition) is satisfied if $\det(\rho_{ij})\neq 0$. Consequently, for appropriate  constants $c_{ij}$, $c_i$ and $c$ such that $\det(\rho_{ij})\neq 0$, the spray $S$ is metrizable. \\

Conversely, let $S$ be metrizable.
Since the condition \textbf{(i)} is satisfied if and only if there exists a locally defined,
0-homogeneous, smooth function $g $ on $\Omega\times\mathbb{R}^n\backslash \{0 \}$, $\Omega$ is open subset of $\mathbb{R}^n$, such that
$$d_J\beta=d_{h_0}g.$$
Then, we have
$$d_J\beta(\partial_i)=d_{h_0}g(\partial_i)\Rightarrow \dot{\partial}_i(b_jy^j)={\partial}_ig\Rightarrow b_i={\partial}_ig.$$
Since $b_i$ is a function of $x$, then $g(x,y)=g_1(x)+g_2(y)$, $g_2(y)$ is 0-homogenous function.
Then, we can write $\beta= y^i b_i(x)=S_0(g)$ and $b_i(x)={\partial}_ig$.
The condition \textbf{(ii)} is satisfied if and only if
$$d_{h_0}\rho - S_0(g)d_J\rho-2\rho d_{h_0}g=0.$$
Applying the above equation on $\partial_i$ and using that $S_0(h)=\beta$ and $\rho=\beta^2-S_0(\beta)$, we have
\begin{equation}
   \label{diff-g(x)}
   \partial_i\rho - \beta \dot{\partial}_i\rho-2\rho \partial_ig=0.
\end{equation}
Making use of $\beta=y^i\partial_i g$ and $\rho=\beta^2-S_0\beta$,  the solution of \eqref{diff-g(x)} is given by
\begin{equation}
   \label{g(x)}
   g(x,y)=-\frac{1}{2}\ln\left(c_{ij}x^ix^j+\langle c',x\rangle+c\right)+g_2(y),
\end{equation}
 Differentiating (\ref{g(x)}) with respect to $\partial_k$ we have
\begin{equation}
   \label{pa-k-g(x)}
   b_k(x)= \partial_kg=-\frac{2c_{ik}x^i+c_k} {2(c_{ij}x^ix^j+\langle c',x\rangle+c)}.
\end{equation}
This completes the proof.
\end{proof}

Making use of the above theorem we have the following

\begin{theorem}\label{projective-flat}
With appropriate constants $c_{ij}$,  $c$, $c'=(c_1, c_2, ..., c_n)$,
the family
\begin{equation}
   \label{general_F}
   F=\sqrt{\frac{4h(x)c_{ij}y^iy^j-4(c_{ij}x^iy^j)^2-4\langle c',y\rangle c_{ij}x^iy^j-\langle c',y\rangle^2} {(2h(x))^2}},
\end{equation}
$h(x)=c_{ij}x^ix^j+\langle c',x\rangle+c$,
is a family of     projectively flat metrics. 
\end{theorem}
\begin{proof}
By differentiating \eqref{general_F} with respect to $x^k$ and $y^k$ respectively, we have
\begin{equation}
   \label{pa-k-F}
   \begin{array}{rcl}
   \partial_k F&=&\frac{1}{2F}\frac{4\left(2c_{rk}x^r+c_k\right)\left(c_{ij}y^iy^j\right)-8c_{ij}c_{kr}x^iy^jy^r-4\langle c',y\rangle c_{kj}y^j}{(2h(x))^2} -\frac{2c_{ik}x^i+c_k}{h(x)}F,
   \end{array}
 \end{equation}
 \begin{equation}
    \label{pa-dot-k-F}
    \begin{array}{rcl}
    \dot{\partial}_k F&=&\frac{1}{2F}\frac{8h(x)c_{ik}y^i-8c_{ij}x^iy^jc_{kr}x^r-4c_kc_{ij}x^iy^j-4\langle c',y\rangle c_{kj}x^j
    -2\langle c',y\rangle c_k}{(2h(x))^2}.
   \end{array}
 \end{equation}
 Again,  differentiating \eqref{pa-dot-k-F} with respect to $x^j$ gives
 
 \begin{equation}
    \label{pa-j-dot-k-F}
    \begin{array}{rcl}
    \partial_j\dot{\partial}_k F&=&-\frac{1}{2F^2}\frac{8h(x)c_{ik}y^i-8c_{ih}c_{kr}x^iy^hx^r-4c_kc_{ih}x^iy^h-4\langle c',y\rangle c_{kh}x^h
    -2\langle c',y\rangle c_k}{(2h(x))^2}\partial_j F\\
    &&+\frac{(2h(x))^2}{2F}\frac{8(2c_{rj}x^r+c_j)c_{ik}y^i-8c_{ih}c_{kj}x^iy^h-8c_{jh}c_{kr}y^hx^r-4c_kc_{jh}y^h-4\langle c',y\rangle c_{jk}}{(2h(x))^4}\\
    &&-\frac{1}{2F}\frac{8h(x)(2c_{rj}x^r+c_j)(8h(x)c_{ik}y^i-8c_{ih}c_{kr}x^iy^hx^r-4c_kc_{ih}x^iy^h-4\langle c',y\rangle c_{kh}x^h
    -2\langle c',y\rangle c_k)}{(2h(x))^4}.
   \end{array}
 \end{equation}
 
 Contracting \eqref{pa-k-F} by $y^k$ and \eqref{pa-j-dot-k-F} by $y^j$ respectively, we get
  \begin{equation}
    \label{y^k-pa-F}
     y^k\partial_k F= -\frac{2c_{ik}x^iy^k+\langle c',y\rangle}{(2h(x))}2F,
    \end{equation}
    
 \begin{equation}
 \label{y^j-pa-j-dot-F}
    \begin{array}{rcl}
  y^j  \partial_j\dot{\partial}_k F&=&-\frac{2}{F}\frac{(2c_{ih}x^iy^h+\langle c',y\rangle) c_{jk}y^j}{(2h(x))^2}-\frac{2}{F}\frac{(2c_{kr}x^r+c_k)}{(2h(x))}\left(\frac{2h(x)c_{jh}y^jy^h-(2c_{ih}x^iy^h+\langle c',y\rangle )^2}{(2h(x))^2}\right)
   \end{array}
 \end{equation}
 
By \eqref{pa-k-F} and \eqref{y^j-pa-j-dot-F}, we find that metric \eqref{general_F} satisfies the system
 $$y^j\dot{\partial}_k\partial_jF=\partial_kF$$
 which assures that $F$ is projectively flat metric.
\end{proof}

\begin{proposition}
The geodesic spray $G^i$ and the Jacobi endomorphism $R^i_{\,\,j}$ of   the metric \eqref{general_F} are given by
$$G^i=-\frac{2c_{ij}x^iy^j+\langle c',y\rangle}{(2h(x))}y^i,$$
$$R^i_{\,\, j}=\frac{4h(x)c_{rs}y^ry^s-(2c_{rs}x^ry^s+\langle c',y\rangle )^2} {(2h(x))^2}\delta^i_j-\frac{4h(x)c_{rj}y^r-(2c_{rs}x^ry^s+4\langle c',y\rangle )(2c_{rj}x^r+c_j)} {(2h(x))^2}y^i.$$
Moreover, the metric \eqref{general_F} has constant flag curvature $1$.
\end{proposition}

\begin{proof}
 By making use of Theorem \ref{projective-flat}, the metric \eqref{general_F} is projectively flat and hence its geodesic spray is given by 
 $$G^i=P(x,y)y^i, \quad P(x,y)=\frac{y^k\partial_kF}{2F}.$$
 Now, using \eqref{y^k-pa-F} we get the required formula for the geodesic spray $G^i$.

The Jacobi endomorphism $R^i_{\,\, j}$ has the form
$$R^i_{\,\, j}=2\partial_jG^i-y^k\partial_kN^i_j+2G^kG^i_{jk}-N^i_kN^k_j$$
where $N^i_j=\dot{\partial}_jG^i$ and $G^i_{jk}=\dot{\partial}_kN^i_j$. By using the formula of $G^i$, we have the following
$$\partial_j G^i=-\frac{(2h(x))2c_{js}y^s-(2c_{rs}x^ry^s+\langle c',y\rangle)(2(2c_{rj}x^r+c_j))}{(2h(x))^2}y^i,$$
$$N^i_j=-\frac{(2c_{rj}x^r+c_j)y^i+(2c_{rs}x^ry^s+\langle c',y\rangle)\delta^i_j}{(2h(x))},$$
$$G^i_{jk}=-\frac{(2c_{rj}x^r+c_j)\delta^i_k+(2c_{rk}x^r+c_k)\delta^i_j}{(2h(x))},$$

$$\partial_kN^i_j=-\frac{(2h(x))(2c_{jk}y^i+2c_{ks}y^s\delta^i_j)-2(2c_{tk}x^t+c_k)((2c_{rj}x^r+c_j)y^i+(2c_{rs}x^ry^s+\langle c',y\rangle)\delta^i_j)}{(2h(x))^2},$$
$$N^i_kN^k_j=\frac{(2c_{rs}x^ry^s+\langle c',y\rangle)(3(2c_{rj}x^r+c_j)y^i+(2c_{rs}x^ry^s+\langle c',y\rangle)\delta^i_j)}{(2h(x))^2},$$
$$G^kG^i_{jk}=\frac{(2c_{rs}x^ry^s+\langle c',y\rangle)((2c_{rj}x^r+c_j)y^i+(2c_{rs}x^ry^s+\langle c',y\rangle)\delta^i_j)}{(2h(x))^2}.$$
By plugging the above quantities into the formula of $R^i_{\,\, j}$, the result follows. 

To show that $F$ has constant flag curvature, we compute the Ricci scalar $Ric$ as follows
$$R^i_{\,\,i}=(n-1)\left(\frac{4h(x)c_{rs}y^ry^s-(2c_{rs}x^ry^s+\langle c',y\rangle )^2} {(2h(x))^2}\right).$$
Hence, the  Ricci scalar is given by $\rho=F^2$ and so $F$  has  constant curvature  $1$.
\end{proof}

As a special case, we have the following
\begin{corollary}
Taking $c_{ij}=\mu\delta_{ij}$,  $c_i=0$, $c=1$, then we get
$$F_\mu=\sqrt{\mu\frac{(1+\mu|x|^2)|y|^2-\mu\langle x,y\rangle^2}{(1+\mu|x|^2)^2}}, \quad G^i=-\frac{\mu\langle x,y \rangle}{1+\mu|x|^2}y^i,$$
 and 
  $$R^i_{\,\, j}=\mu\left(\frac{(1+\mu|x|^2)|y|^2-\mu\langle x,y\rangle^2}{(1+\mu|x|^2)^2}\delta^i_j-\frac{(1+\mu|x|^2)\delta_{rj}y^r-\mu(\delta_{rs}x^ry^s)(\delta_{kj}x^k)} {(1+\mu|x|^2)^2}y^i\right).$$
$$R^i_{\,\, i}=(n-1)\mu\left(\frac{(1+\mu|x|^2)|y|^2-\mu\langle x,y\rangle^2}{(1+\mu|x|^2)^2}\right).$$
\end{corollary}

\begin{corollary}
   \label{c_ijnon0}
   Let $S=S_0-2\beta \C$ be a one form deformation of a flat spray $S_0$,
$$\beta=-\frac{2c_{ij}x^iy^j+\langle c',y\rangle} {2(c_{ij}x^ix^j+\langle c',x\rangle+c)}.$$
A necessary condition for $S$ to be metrizable  is ${c_{ij}}\neq0$. 
\end{corollary}

 It is known that, \cite{Bucataru3},  one of the conditions for   a spray $S$  with non-vanishing Ricci curvature to be  metrizable by a Finsler function of non-zero constant flag curvature is  $\text{rank}\ dd_J(\text{Tr} \ \Phi)=2n$.
As an application of the deformation of a flat spray by a  one form, we answer the following question: \\

\textit{ Does  any spray of non-vanishing Ricci curvature  satisfy the condition $\text{rank}\ dd_J(\text{Tr} \ \Phi)=2n$?}\\
 
The following proposition shows that, for a spray $S$,  if  $S$ has non vanishing Ricci curvature, then the  rank of the form $ dd_J(\text{Tr} \ \Phi)$ not necessarily maximal; that is, the condition $\text{rank}\ dd_J(\text{Tr} \ \Phi)=2n$ is sharp for the metrizability of $S$.

\begin{proposition} Let $\Phi$ the Jacobi endomorphism of  a spray $S$ , then we have:
\begin{description}
  \item[(a)] If $\text{rank}\ dd_J(\text{Tr} \ \Phi)=2n$, then $S$ has non-vanishing Ricci curvature.
  \item[(b)] If $S$ has non-vanishing Ricci curvature, then $\text{rank}\ dd_J(\text{Tr} \ \Phi)$ is not necessarily   maximal.
\end{description}
\end{proposition}
\begin{proof}The proof of \textbf{(a)} is obvious, so we prove \textbf{(b)} only. The proof of  \textbf{(b)} can be performed by providing an example in which $S$ has non-vanishing Ricci curvature and $ dd_J(\text{Tr} \ \Phi)$ has not maximal rank.
Let $$S=S_0-2\beta \C, \quad \beta=-\frac{\langle c',y\rangle} {2(\langle c',x\rangle+c)}, \quad b_i(x)=-\frac{c_i} {2(\langle c',x\rangle+c)},$$
where $c'=(c_1,c_2,...,c_n)$, $c_i$ and $c$ are arbitrary  constants.
Since $\text{Tr} \ \Phi=\text{Ric}=(n-1)(\beta^2-S_0\beta)$, Ric is the Ricci curvature. Then, we get
$$\text{Ric}=-(n-1)\frac{\langle c',y\rangle^2} {4(\langle c',x\rangle+c)^2}.$$
Since  $n\neq 1$, we have $\text{Ric}\neq 0$. Straightforward calculations lead to
$$dd_J(\text{Tr} \ \Phi)=2(n-1)(\alpha_{ij}dx^i\wedge dy^j+\beta_{ij}dx^i\wedge dx^j),$$
where $\alpha_{ij}=\frac{c_ic_j}{4(\langle c',x\rangle+c)^2}$. Then,
$ dd_J(\text{Tr} \ \Phi)$ has  maximal rank if $\det(\alpha_{ij})\neq 0$, but $\det(\alpha_{ij})= 0$ and moreover $\text{rank}(\alpha_{ij})=1$, then the result follows.
\end{proof}

\section{Affine transformations of Klein metric}

The  metric on $\mathbb{B}^n\subset\mathbb{R}^n$  given by

$$ F_{\mu}=\sqrt{\frac{(1+\mu|x|^2)|y|^2-\mu\langle x,y\rangle^2}{(1+\mu|x|^2)^2}},\quad P=\frac{-\mu\langle x,y \rangle}{1+\mu|x|^2}, \qquad y\in T_x\mathbb{B}^n\simeq \mathbb{R}^n $$
 is projectively flat Riemannian metrics of constant (flag) curvature $\mu$. It is known that  every locally projectively flat Riemannian metric is locally isometric to $F_{\mu}$ for some constant $\mu$.   Starting by the Klein metric, the affine transformation $x$ to $Ax+B$ and $y$ to $Ay$ where $A$ is an $n\times n$ invertible matrix and $B$ is an arbitrary $n\times1$ matrix, produces projectively flat Riemannian metrics.  
 
\medskip 
 
 In this section, taking $\mu=1$, we show that the family \eqref{general_F} is not isometric to the Klein metric via affine transformations.

\begin{theorem}
The family \eqref{general_F} is not, generally,  isometric to  the Klein metric $(\mu=1)$ via affine transformations. 
\end{theorem}
\begin{proof}
Consider the affine  transformation  $\overline{x}=Ax+B$ and $\overline{y}=Ay$, where $A$ is an $n\times n$ invertible matrix and $B$ is an arbitrary $n\times1$ matrix,
$$A=\left(    \begin{array}{cccc}         
                 a_{11}   & \ldots & a_{1n}\\
                \vdots  & \ddots & \vdots\\
                a_{n1}         &\ldots & a_{nn}  \end{array}\right),\qquad 
               B=\left( \begin{array}{cccc} 
               b_1\\
               \vdots\\
                b_n\end{array}\right).$$
  Then we have
  $$|\overline{x}|^2=(\sum_{k=1}^na_{ki}a_{kj})x^ix^j+2(\sum_{k=1}^na_{ki}b_{k})x^i+|b|^2,$$
  $$|\overline{y}|^2=(\sum_{k=1}^na_{ki}a_{kj})y^iy^j,$$
  $$\langle \overline{x},\overline{y}\rangle=(\sum_{k=1}^na_{ki}a_{kj})x^iy^j+(\sum_{k=1}^na_{ki}b_{k})x^i,$$              
 where $|b|^2=b_1^2+b_2^2+...+b_n^2$. Therefore, the Klein metric transforms to 
\begin{equation}
\label{F_Klein_trans}
F=\sqrt{\frac{H(x)(\sum_{k=1}^na_{ki}a_{kj})y^iy^j-((\sum_{k=1}^na_{ki}a_{kj})x^iy^j)^2-2\langle B',y\rangle (\sum_{k=1}^na_{ki}a_{kj})x^iy^j-\langle B',y\rangle^2} {(H(x))^2}},
\end{equation}
where $H(x):= 1+((\sum_{k=1}^na_{ki}a_{kj})x^ix^j+2(\sum_{k=1}^na_{ki}b_{k})x^i+|b|^2)$,  \,\,$B'=(B_1,...,B_n)$,  $B_i=\sum_{k=1}^na_{ki}b_{k}$.
Now, comparing the equations \eqref{general_F} and \eqref{F_Klein_trans}, we get
 $$
  2c_{ij} =\sum_{k=1}^na_{ki}a_{kj}, \qquad
   c_{i} =\sum_{k=1}^na_{ki}b_{k}, \qquad
  2c = 1+|b|^2.
$$
Thus we get, formally, the class \eqref{general_F} provided that the above system is consistent. But, generally, the above system is inconsistent. So once you have the transformation then you  get the $c's$, but if you have the $c's$ then the transformation not necessarily  exist. 

For example, 
let $$A=\left(
      \begin{array}{cc}
        a_{11} & a_{12} \\
        a_{21} & a_{22} \\
      \end{array}
    \right),\qquad B=\left(
                      \begin{array}{c}
                        b_1 \\
                        b_2 \\
                      \end{array}
                    \right).
    $$
   Then the constants $c's$ are given by
\begin{align*}
  c_{11} &= \frac{1}{2}( a_{11}^2+a_{21}^2 )\\
  c_{22} &= \frac{1}{2}(  a_{12}^2+a_{22}^2 )\\
  c_{12} &=c_{21}=  a_{11}a_{12}+a_{21}a_{22} \\
  c_1 &= a_{11}b_1+a_{21}b_2 \\
  c_2 &=a_{12} b_1+a_{22}b_2 \\
  c &= \frac{1}{2}(1+(b_1^2+b_2^2)).
\end{align*}
Now, take $c_{11}=c_{22}=0, c_{12}=c_{21}=1, c_1=1,c_2=1,c=1$, we get  a projectively flat metric and at the same time by substitution in the above system one obtains inconsistent system. For this choice of the $c's$, we have
$$F=\sqrt{\frac{(8x_1x_2y_1y_2-4x_1^2y_2^2-4x_2^2y_1^2+4x_1y_1y_2-4x_1y_2^2-4x_2y_1^2+4x_2y_1y_2-y_1^2+6y_1y_2-y_2^2)}{4(2x_1x_2+x_1+x_2+1)^2}},$$
and the projective factor is given by
$$\beta=-\frac{1}{2}\frac{2x_1y_2+2x_2y_1+y_1+y_2}{2x_1x_2+x_1+x_2+1}.$$
So one can say that the affine transformation of Klein metric is contained in \eqref{general_F} but not any metric in \eqref{general_F} can be obtained by an affine transformation.
\end{proof}

Now, the question is \\

\textit{What is the isometry  (transformation) between the klein metric and  the family \eqref{general_F}?
}
 \section{Finsler solutions for Hilbert fourth problem and examples}

Since the deformation spray $S$ of a flat spray $S_0$ is always isotropic and in the case which the curvature of $S$ is non zero, then the metric freedom \cite{Mu-Elgendi} of $S$ is unique up to some constants. Therefore, in our case the deformation of a flat spray  by the specific  one form $\beta=b_i(x)y^i$ where $b_i(x)$ given by \eqref{b_i(x)} is metrizable by unique Riemannain metric given in \eqref{general_F}. However, in this section, we introduce  some new projectively flat Finsler metrics and hence new Finsler solutions for Hilbert's fourth problem. Although,  Lots of new projectively flat Finsler metrics can be constructed, we will mention only two examples. 

\medskip

For simplicity we consider  the following  special case.
 \begin{corollary}
 Putting $c_{ij}=\lambda\delta_{ij}$,    we have
 \begin{equation}
 \label{F_proj_flat}
F=\sqrt{\frac{4\lambda(\lambda|x|^2+\langle c',x \rangle+c)|y|^2-4\lambda^2\langle x,y\rangle^2-4\lambda\langle c',y \rangle\langle x,y \rangle -\langle c',y \rangle^2}{4(\lambda|x|^2+\langle c',x \rangle+c)^2}},
\end{equation}
is a family of projectively flat metrics with the projective factor
$$\beta=-\frac{2\lambda\langle x,y \rangle+\langle c',y \rangle}{2(\lambda|x|^2+\langle c',x \rangle+c)}.$$
\end{corollary}

By making use of the above corollary, since $\beta$ is closed one form on $M$ and $F$ is projectively flat Riemannian metric, then we have the following  example of projectively flat Finsler metric.

\begin{example}
The family of metrics
 $$\overline{F}=\frac{\sqrt{4\lambda (\lambda|x|^2+\langle c',x \rangle+c)|y|^2-4\lambda^2\langle x,y\rangle^2-4\lambda\langle c',y \rangle\langle x,y \rangle -\langle c',y \rangle^2}+ (2\lambda\langle x,y \rangle+\langle c',y \rangle )}{2(\lambda|x|^2+\langle c',x \rangle+c)}$$
 is new family of projectively flat Finsler metrics. Where
$$\overline{G}^i=P(x,y)y^i, \quad P(x,y)=-\frac{2\lambda\langle x,y \rangle+\langle c',y \rangle}{2(\lambda|x|^2+\langle c',x \rangle+c)}+\left(F-\frac{\lambda|y|^2}{2F(\lambda|x|^2+\langle c',x \rangle+c)}\right).$$
Consequently, we have new Finsler solutions for Hilbert's fourth problem.
\end{example}

By the help of \cite{Shen-book} (Example 8.2.2, Page 156), we have another Finsler solution for   Hilbert's fourth problem as follows.

\begin{example}
The metric
\begin{equation*}
\label{F_Funk}
\begin{split}
&\Theta(x,y)=\frac{\langle c',y\rangle\langle c',x\rangle-4\lambda \langle x,y\rangle} {4c\lambda |x|^2-c^2}\\
&+\frac{\sqrt{16 \lambda^2c^2(\langle x,y\rangle^2-|x|^2|y|^2)+\langle c',y\rangle^2(\langle c',x\rangle^2+4\lambda |x|^2-c^2)-8\lambda c\langle c',y\rangle\langle c',x\rangle\langle x,y\rangle+4\lambda c^3|y|^2}}{4c\lambda |x|^2-c^2}æ
\end{split}
\end{equation*}

is Funk metric  and, moreover, it is projectively flat with the projective factor
$P=\frac{\Theta(x,y)}{2}$. Thus $\Theta(x,y)$ is projectively flat with constant flag curvature
$-\frac{1}{4}$.
\end{example}

\begin{proof}
Using  \eqref{F_proj_flat}, we have 
$$F(0,y)=\phi(y)=\sqrt{\frac{4\lambda c |y|^2-\langle c',y\rangle^2}{4c^2}}.$$
Define 
$$\Theta(x,y)=\phi(y+\Theta(x,y)x)=\sqrt{\frac{4\lambda c |y+\Theta(x,y)x|^2-\langle c',y+\Theta(x,y)x\rangle^2}{4c^2}}.$$
Squaring both sides of the above  equation and solving it for $\Theta$, we get the required formula.
Since  $\phi(y)$ is a Minkowski norm, then $\Theta(x,y)$ is Funk metric and it is projectively flat metric with the projective factor $P=\frac{\Theta(x,y)}{2}$.
\end{proof}



The following example shows a  one form deformation of a  non flat spray which is not metrizable.

\begin{example}[] \ %
  \\ Let $M= \{(x^1,x^2)\in\mathbb{R}^2:\, x^2> 2\}$ and $S_0$ be a spray
  given by the coefficients
  \begin{displaymath}
    G^1_0:=   \frac{(y^1)^2}{2x^2}, \qquad
    G^2_0:=0,
  \end{displaymath}
  and take $\beta=y^1+y^2$.
  Now consider the deformation $S=S_0-2\beta C$. The new coefficients are given by
   \begin{displaymath}
    G^1:=   \frac{(y^1)^2}{2x^2}+y^1(y^1+y^2), \qquad
    G^2:=y^2(y^1+y^2),
  \end{displaymath}
  The spray $S$ is isotropic and the coefficients of the nonlinear connection
  are given by
  \begin{displaymath}
    N_1^1=\frac{y^1}{x^2}+2y^1+y^2,\quad
    N_1^2=y^1,\quad
    N_2^1=y^2, \quad
    N_2^2=y^1+2y^2.
  \end{displaymath}
  The horizontal basis is $\{h_1,h_2\}$ where
  \begin{alignat*}{1}
    h_1&=\frac{\partial}{\partial x^1}-\left(\frac{y^1}{x^2}+2y^1+y^2\right)\frac{\partial}{\partial
      y^1}-y^2\frac{\partial}{\partial
      y^2},
    \\
    h_2&=\frac{\partial}{\partial x^2}-y^1\frac{\partial}{\partial
      y^1}-(y^1+2y^2)\frac{\partial}{\partial
      y^2}.
  \end{alignat*}
  We have  
   \begin{alignat*}{1}
    v_{1}:=[[h_1,h_2],h_1]&=-\left(\frac{((x^2)^2y^1+(x^2)^2y2+x^2y^1-y^2x^2+y^1}{(x^2)^2}\right)\frac{\partial}{\partial y^1}\\
   & +\left(\frac{(x^2)^2y^1(x^2)^2y^2+2x^2y^1+x^2y^2+y^1}{(x^2)^2}\right)\frac{\partial}{\partial
        y^2}
    \\
    v_{2}:=\big[[h_1,h_2],h_2\big]& =-\left(\frac{(x^2)^3y^1+(x^2)^3y^2+2y^1}{(x^2)^3}
   \right)\frac{\partial}{\partial y^1}\\
 &  +\left(\frac{(x^2)^2y^1+(x^2)^2y^2-x^2y^1+2y^1}{(x^2)^2}\right)\frac{\partial}{\partial
        y^2}.
  \end{alignat*}

  Being $v_1$ and $v_2$ linearly independent we have
  $\mathcal{H} =Span\{h_1, h_2, v_1, v_2\}= T\TM$, where $\mathcal{H}$ is the holonomy distribution generated by  the horizontal vectors and their successive Lie brackets. Consequently, the Liouville  vector field $\C \in \mathcal{H}$ hence the spray is not metrizable.

\end{example}
The following example introduces  a one form deformation of a flat spray which is not   metrizable.
\begin{example}[] \ %
  \\ Let $M= \{(x^1,x^2)\in\mathbb{R}^2:\, x^2>0\}$ and $S_0$ be a flat  spray.
 So  the coefficients are given by  
  \begin{displaymath}
    G_0^1= 
    G_0^2=0,
  \end{displaymath}
  and take $\beta=y^1+y^2$.
  Now consider the deformation $S=S_0-2\beta \C$. The new coefficients are given by
   \begin{displaymath}
    G^1:=   y^1(y^1+y^2), \qquad
    G^2:=y^2(y^1+y^2),
  \end{displaymath}
  The spray $S$ is isotropic and the coefficients of the nonlinear connection
  are given by
  \begin{displaymath}
    N_1^1=2y^1+y^2,\quad
    N_1^2=y^1,\quad
    N_2^1=y^2, \quad
    N_2^2=y^1+2y^2.
  \end{displaymath}
  The horizontal basis is $\{h_1,h_2\}$ where
  \begin{alignat*}{1}
    h_1&=\frac{\partial}{\partial x^1}-\left(2y^1+y^2\right)\frac{\partial}{\partial
      y^1}-y^2\frac{\partial}{\partial
      y^2},
    \\
    h_2&=\frac{\partial}{\partial x^2}-y^1\frac{\partial}{\partial
      y^1}-(y^1+2y^2)\frac{\partial}{\partial
      y^2}.
  \end{alignat*}
  We have
  \begin{alignat*}{1}
    v_{1}:=[h_1,h_2]&=-\left(y^1+y^2\right)\frac{\partial}{\partial y^1}+\left(y^1+y^2\right)\frac{\partial}{\partial
        y^2}.
  \end{alignat*}
  The successive Lie brackets of $h_1$ and $h_2$ produce no more linearly independent vectors and hence the holonomy distribution   $\mathcal{H} =Span\{h_1, h_2, v_1\}$. The metric freedom \cite{Mu-Elgendi} of $S$ is unique.  Now we can check if we have  regular energy function metricizes $S$ or not.
 
 The spray ${S}$ is Finsler metrizable  if there exists a function  ${E}$ satisfying
the following system of partial differential equations
$$\mathcal{L}_C{E}=2{E}, \hspace{1cm } d_{{h}}{E}=0,$$
which can be written in the form
$$y_1\dot{\partial}_1{E}+y_2\dot{\partial}_2 {E}-2{E}=0,$$
$$  \frac{\partial E}{\partial x^1}-\left(2y^1+y^2\right)\frac{\partial E}{\partial
      y^1}-y^2\frac{\partial E}{\partial
      y^2}=0,$$
$$\frac{\partial E}{\partial x^2}-y^1\frac{\partial E}{\partial
      y^1}-(y^1+2y^2)\frac{\partial E}{\partial
      y^2} =0,$$
 $$-\left(y^1+y^2\right)\frac{\partial E}{\partial y^1}+\left(y^1+y^2\right)\frac{\partial E}{\partial
        y^2}=0.$$
 The above system has the solution
 $$E=C_1e^{4(x^1+x^2)}(y^1+y^2)^2.$$
 The matrix $(g_{ij})$ associated with $E$ is singular and hence the spray is not metrizable. 
 Here in this example $\beta$ is closed but  $det(\partial_ib_j+b_ib_j)=0$.
\end{example}

\section*{Acknowledgement} The author would like to express his deep  gratitude to Professor   Zolt\'{a}n Muzsnay (University of Debrecen) for his valuable discussions and  suggestions.

\providecommand{\bysame}{\leavevmode\hbox
to3em{\hrulefill}\thinspace}
\providecommand{\MR}{\relax\ifhmode\unskip\space\fi MR }
\providecommand{\MRhref}[2]{%
  \href{http://www.ams.org/mathscinet-getitem?mr=#1}{#2}
} \providecommand{\href}[2]{#2}

\end{document}